\documentclass[leqno,epsf,12pt]{amsart}
\usepackage{amsmath,epsfig}
\usepackage{amssymb}
\usepackage{amsfonts,amsthm,amscd}
\usepackage{hyperref}

\numberwithin{equation}{section}

\newtheorem{theorem}{Theorem}[section]
\newtheorem{corollary}[theorem]{Corollary}
\newtheorem{lemma}[theorem]{Lemma}
\newtheorem{proposition}[theorem]{Proposition}
\theoremstyle{definition}
\newtheorem{defn}[theorem]{Definition}

\theoremstyle{remark}
\newtheorem{remark}[theorem]{Remark}

\newcommand{\nc}{\newcommand}
\newcommand{\be}{\begin{equation}}
\newcommand{\ee}{\end{equation}}

\newcommand{\bc}{\begin{center}}
\newcommand{\ec}{\end{center}}
\nc{\bth}{\begin{theorem}} \nc{\bpr}{\begin{proposition}}
\nc{\epr}{\end{proposition}} \nc{\ble}{\begin{lemma}}
\nc{\ele}{\end{lemma}} \nc{\bco}{\begin{corollary}}
\nc{\eco}{\end{corollary}} \nc{\bre}{\begin{remark}}
\nc{\ere}{\end{remark}}
   \nc{\f}{\frac}
   \nc{\pa}{\partial}
   \nc{\na}{\nabla}
   \nc{\al}{\alpha}
   \nc{\bet}{\beta}
   \nc{\ga}{\gamma}
   \nc{\de}{\delta_1}
   \nc{\De}{\Delta}
   \nc{\Om}{\Omega}
   \nc{\om}{\omega}
   \nc{\tom}{\tilde\omega}
   \nc{\ep}{\varepsilon}
   \nc{\tvp}{\tilde\varphi}
   \nc{\vp}{\varphi}
   \nc{\R}{\mathbb R}
   \nc{\Z}{\mathbb Z}
   \nc{\CC}{\mathbb C}
   \nc{\Ad}{\rm Ad}
   \nc{\fraka}{{\mathfrak g}_{\mathcal A}}
   \nc{\Ga}{G_{\mathcal A}}
\usepackage[all]{xy}               
\usepackage{exscale}
\usepackage{color}
\usepackage{axodraw}
\usepackage{colordvi}

\def\at1{\begin{array}{c} \ta1\ \\ \end{array}}
\def\Mat31{\begin{array}{c} \td31\ \\ \end{array}}
\def\mat41{\begin{array}{c} \tb2\ \\ \end{array}}
\def\mot43{\begin{array}{c} \th43\ \\ \end{array}}

\def\ta1{{\scalebox{0.25}{ 
\begin{picture}(12,12)(38,-38)
\SetWidth{0.5} \SetColor{Black} \Vertex(45,-33){5.66}
\end{picture}}}}

\def\tb2{{\scalebox{0.25}{ 
\begin{picture}(12,42)(38,-38)
\SetWidth{0.5} \SetColor{Black} \Vertex(45,-3){5.66}
\SetWidth{1.0} \Line(45,-3)(45,-33) \SetWidth{0.5}
\Vertex(45,-33){5.66}
\end{picture}}}}

\def\tc3{{\scalebox{0.25}{ 
\begin{picture}(12,72)(38,-38)
\SetWidth{0.5} \SetColor{Black} \Vertex(45,27){5.66}
\SetWidth{1.0} \Line(45,27)(45,-3) \SetWidth{0.5}
\Vertex(45,-33){5.66} \SetWidth{1.0} \Line(45,-3)(45,-33)
\SetWidth{0.5} \Vertex(45,-3){5.66}
\end{picture}}}}

\def\td31{{\scalebox{0.25}{ 
\begin{picture}(42,42)(23,-38)
\SetWidth{0.5} \SetColor{Black} \Vertex(45,-3){5.66}
\Vertex(30,-33){5.66} \Vertex(60,-33){5.66} \SetWidth{1.0}
\Line(45,-3)(30,-33) \Line(60,-33)(45,-3)
\end{picture}}}}

\def\te4{{\scalebox{0.25}{ 
\begin{picture}(12,102)(38,-8)
\SetWidth{0.5} \SetColor{Black} \Vertex(45,57){5.66}
\Vertex(45,-3){5.66} \Vertex(45,27){5.66} \Vertex(45,87){5.66}
\SetWidth{1.0} \Line(45,57)(45,27) \Line(45,-3)(45,27)
\Line(45,57)(45,87)
\end{picture}}}}

\def\tf41{{\scalebox{0.25}{ 
\begin{picture}(42,72)(38,-8)
\SetWidth{0.5} \SetColor{Black} \Vertex(45,27){5.66}
\Vertex(45,-3){5.66} \SetWidth{1.0} \Line(45,27)(45,-3)
\SetWidth{0.5} \Vertex(60,57){5.66} \SetWidth{1.0}
\Line(45,27)(60,57) \SetWidth{0.5} \Vertex(75,27){5.66}
\SetWidth{1.0} \Line(75,27)(60,57)
\end{picture}}}}

\def\tg42{{\scalebox{0.25}{ 
\begin{picture}(42,72)(8,-8)
\SetWidth{0.5} \SetColor{Black} \Vertex(45,27){5.66}
\Vertex(45,-3){5.66} \SetWidth{1.0} \Line(45,27)(45,-3)
\SetWidth{0.5} \Vertex(15,27){5.66} \Vertex(30,57){5.66}
\SetWidth{1.0} \Line(15,27)(30,57) \Line(45,27)(30,57)
\end{picture}}}}

\def\th43{{\scalebox{0.25}{ 
\begin{picture}(42,42)(8,-8)
\SetWidth{0.5} \SetColor{Black} \Vertex(45,-3){5.66}
\Vertex(15,-3){5.66} \Vertex(30,27){5.66} \SetWidth{1.0}
\Line(15,-3)(30,27) \Line(45,-3)(30,27) \Line(30,27)(30,-3)
\SetWidth{0.5} \Vertex(30,-3){5.66}
\end{picture}}}}

\def\thj44{{\scalebox{0.25}{ 
\begin{picture}(42,72)(8,-8)
\SetWidth{0.5} \SetColor{Black} \Vertex(30,57){5.66}
\SetWidth{1.0} \Line(30,57)(30,27) \SetWidth{0.5}
\Vertex(30,27){5.66} \SetWidth{1.0} \Line(45,-3)(30,27)
\SetWidth{0.5} \Vertex(45,-3){5.66} \Vertex(15,-3){5.66}
\SetWidth{1.0} \Line(15,-3)(30,27)
\end{picture}}}}

\def\ti5{{\scalebox{0.25}{ 
\begin{picture}(12,132)(23,-8)
\SetWidth{0.5} \SetColor{Black} \Vertex(30,117){5.66}
\SetWidth{1.0} \Line(30,117)(30,87) \SetWidth{0.5}
\Vertex(30,87){5.66} \Vertex(30,57){5.66} \Vertex(30,27){5.66}
\Vertex(30,-3){5.66} \SetWidth{1.0} \Line(30,-3)(30,27)
\Line(30,27)(30,57) \Line(30,87)(30,57)
\end{picture}}}}

\def\tj51{{\scalebox{0.25}{ 
\begin{picture}(42,102)(53,-38)
\SetWidth{0.5} \SetColor{Black} \Vertex(61,27){4.24}
\SetWidth{1.0} \Line(75,57)(90,27) \Line(60,27)(75,57)
\SetWidth{0.5} \Vertex(90,-3){5.66} \Vertex(60,27){5.66}
\Vertex(75,57){5.66} \Vertex(90,-33){5.66} \SetWidth{1.0}
\Line(90,-33)(90,-3) \Line(90,-3)(90,27) \SetWidth{0.5}
\Vertex(90,27){5.66}
\end{picture}}}}

\def\tk52{{\scalebox{0.25}{ 
\begin{picture}(42,102)(23,-8)
\SetWidth{0.5} \SetColor{Black} \Vertex(60,57){5.66}
\Vertex(45,87){5.66} \SetWidth{1.0} \Line(45,87)(60,57)
\SetWidth{0.5} \Vertex(30,57){5.66} \SetWidth{1.0}
\Line(30,57)(45,87) \SetWidth{0.5} \Vertex(30,-3){5.66}
\SetWidth{1.0} \Line(30,-3)(30,27) \SetWidth{0.5}
\Vertex(30,27){5.66} \SetWidth{1.0} \Line(30,57)(30,27)
\end{picture}}}}

\def\tl53{{\scalebox{0.25}{ 
\begin{picture}(42,102)(8,-8)
\SetWidth{0.5} \SetColor{Black} \Vertex(30,57){5.66}
\Vertex(30,27){5.66} \SetWidth{1.0} \Line(30,57)(30,27)
\SetWidth{0.5} \Vertex(30,87){5.66} \SetWidth{1.0}
\Line(30,27)(45,-3) \SetWidth{0.5} \Vertex(15,-3){5.66}
\SetWidth{1.0} \Line(15,-3)(30,27) \Line(30,57)(30,87)
\SetWidth{0.5} \Vertex(45,-3){5.66}
\end{picture}}}}

\def\tm54{{\scalebox{0.25}{ 
\begin{picture}(42,72)(8,-38)
\SetWidth{0.5} \SetColor{Black} \Vertex(30,-3){5.66}
\SetWidth{1.0} \Line(30,27)(30,-3) \Line(30,-3)(45,-33)
\SetWidth{0.5} \Vertex(15,-33){5.66} \SetWidth{1.0}
\Line(15,-33)(30,-3) \SetWidth{0.5} \Vertex(45,-33){5.66}
\SetWidth{1.0} \Line(30,-33)(30,-3) \SetWidth{0.5}
\Vertex(30,-33){5.66} \Vertex(30,27){5.66}
\end{picture}}}}

\def\tn55{{\scalebox{0.25}{ 
\begin{picture}(42,72)(8,-38)
\SetWidth{0.5} \SetColor{Black} \Vertex(15,-33){5.66}
\Vertex(45,-33){5.66} \Vertex(30,27){5.66} \SetWidth{1.0}
\Line(45,-33)(45,-3) \SetWidth{0.5} \Vertex(45,-3){5.66}
\Vertex(15,-3){5.66} \SetWidth{1.0} \Line(30,27)(45,-3)
\Line(15,-3)(30,27) \Line(15,-3)(15,-33)
\end{picture}}}}

\def\tp56{{\scalebox{0.25}{ 
\begin{picture}(66,111)(0,0)
\SetWidth{0.5} \SetColor{Black} \Vertex(30,66){5.66}
\Vertex(45,36){5.66} \SetWidth{1.0} \Line(30,66)(45,36)
\Line(15,36)(30,66) \SetWidth{0.5} \Vertex(30,6){5.66}
\Vertex(60,6){5.66} \SetWidth{1.0} \Line(60,6)(45,36)
\SetWidth{0.5}
\SetWidth{1.0} \Line(45,36)(30,6) \SetWidth{0.5}
\Vertex(15,36){5.66}
\end{picture}}}}

\def\tq57{{\scalebox{0.25}{ 
\begin{picture}(81,111)(0,0)
\SetWidth{0.5} \SetColor{Black} \Vertex(45,36){5.66}
\Vertex(30,6){5.66} \Vertex(60,6){5.66} \SetWidth{1.0}
\Line(60,6)(45,36) \SetWidth{0.5}
\SetWidth{1.0} \Line(45,36)(30,6) \SetWidth{0.5}
\Vertex(75,36){5.66} \SetWidth{1.0} \Line(45,36)(60,66)
\Line(60,66)(75,36) \SetWidth{0.5} \Vertex(60,66){5.66}
\end{picture}}}}

\def\tr58{{\scalebox{0.25}{ 
\begin{picture}(81,111)(0,0)
\SetWidth{0.5} \SetColor{Black} \Vertex(60,6){5.66}
\Vertex(75,36){5.66} \SetWidth{1.0} \Line(60,66)(75,36)
\SetWidth{0.5} \Vertex(60,66){5.66}
\SetWidth{1.0} \Line(60,36)(60,66) \Line(60,6)(60,36)
\SetWidth{0.5} \Vertex(60,36){5.66} \Vertex(45,36){5.66}
\SetWidth{1.0} \Line(60,66)(45,36)
\end{picture}}}}

\def\ts59{{\scalebox{0.25}{ 
\begin{picture}(81,111)(0,0)
\SetWidth{0.5} \SetColor{Black}
\Vertex(75,36){5.66} \SetWidth{1.0} \Line(60,66)(75,36)
\SetWidth{0.5} \Vertex(60,66){5.66}
\SetWidth{1.0} \Line(60,36)(60,66) \SetWidth{0.5}
\Vertex(60,36){5.66} \Vertex(45,36){5.66} \SetWidth{1.0}
\Line(60,66)(45,36) \Line(75,6)(75,36) \SetWidth{0.5}
\Vertex(75,6){5.66}
\end{picture}}}}

\def\tz591{{\scalebox{0.25}{ 
\begin{picture}(81,111)(0,0)
\SetWidth{0.5} \SetColor{Black}
\Vertex(75,36){5.66} \SetWidth{1.0} \Line(60,66)(75,36)
\SetWidth{0.5} \Vertex(60,66){5.66}
\SetWidth{1.0} \Line(60,36)(60,66) \SetWidth{0.5}
\Vertex(60,36){5.66} \Vertex(45,36){5.66} \SetWidth{1.0}
\Line(60,66)(45,36) \SetWidth{0.5} \Vertex(45,6){5.66}
\SetWidth{1.0} \Line(45,6)(45,36)
\end{picture}}}}

\begin{document}

\normalsize
\title[Completely Integrable Lax Pair equations]{Completely Integrable Lax Pair equations for Connes-Kreimer Hopf algebra of rooted trees}
\date{}
\author{Gabriel B\u adi\c toiu}
\address{Institute of Mathematics ``Simion Stoilow'' of the Romanian Academy, Research Unit No. 4, P.O. Box 1-764,
014700 Bucharest, Romania.
 {\tt Gabriel.Baditoiu@imar.ro}}
\begin{abstract}
   We show the complete integrability of the Lax pair equations for certain low-dimensional Lie algebras of the infinitesimal character $\tilde\beta_0$ introduced in the paper \emph{Lax pair equations and {C}onnes-{K}reimer renormalization}, by  G. B{\u{a}}di{\c{t}}oiu and S. Rosenberg, Comm. Math. Phys. \textbf{296} (2010), 655--680
\end{abstract}

\maketitle

\section{Introduction}

Many authors used group-theoretical methods to construct completely integrable Hamiltonian systems and their solutions (see the survey due to Reyman and Semenov-Tian-Shansky in \cite{sts}).
In \cite{baditoiu}, Steven Rosenberg and I constructed a Lax pair equation associated to the Connes-Kreimer-Birkhoff factorization of the character group of a connected graded commutative Hopf algebra and we gave a completely integrable Lax pair equation example (see \S 8.3.2 in \cite{baditoiu}). In this note, I show that the  Lax pair equations of a certain infinitesimal character $\tilde\beta_0$ (introduced in \cite{baditoiu}, see \S \ref{prelim}) in low dimensions are completely integrable.
In section \ref{prelim}, we recall some of the  Lax pair equations introduced in \cite{baditoiu} and some basic notion in the Connes-Kreimer renormalization. In \S\ref{main}, we show that the Lax pairs associated to the infinitesimal character $\tilde\beta_0$ (from \cite{baditoiu}) for certain low dimensional examples are completely integrable in the double Lie algebras.

\section{Preliminaries}\label{prelim}
We now briefly recall the some notion from Connes-Kreimer renormalizations and the Lax pair equations introduced in \cite{baditoiu}.

Let ${\mathcal H}=({\mathcal H}, 1, \mu,\De,\ep,S)$ be a graded
connected Hopf algebra over ${\mathbb C}$. Let $\mathcal A$ be a
unital commutative algebra with  unit $1_\mathcal A$.  Important choices for $\mathcal A$ are:
(i) $\mathcal{A}$  be the algebra of Laurent series $\CC[\lambda^{-1},\lambda]]$; or (ii)  ${\mathcal A} = \CC.$ (see \cite{man}).

\begin{defn}(see \cite{ck1,man,baditoiu})
  (i) The {\bf character group} $G_\mathcal A$ of the  Hopf algebra ${\mathcal H}$ is
  the set of algebra morphisms $\phi:{\mathcal H}\to\mathcal{A}$
  with $\phi(1)=1_\mathcal{A}.$
  The group law is given by the convolution product
    $$
      (\psi_1\star\psi_2)(h)=\langle  \psi_1\otimes\psi_2,\De h\rangle;$$
  the unit element
  is $\ep$.

(ii)
  An $\mathcal A$-valued {\bf infinitesimal character} of a Hopf algebra ${\mathcal H}$ is a
  ${\mathbb C}$-linear map $Z:{\mathcal H}\to\mathcal{A}$ satisfying
  $$\langle Z,hk\rangle =\langle Z, h\rangle \varepsilon(k)+\varepsilon(h)
  \langle Z,k\rangle.$$ The set of infinitesimal characters is denoted
  by $\mathfrak{g}_\mathcal{A}$ and is endowed with a Lie algebra bracket:
  $$[Z,Z']=Z\star Z'-Z'\star Z,\ \  \mathrm{for\ }Z,\ Z'\in\mathfrak{g}_\mathcal{A}, $$
  where
  $\langle Z\star Z',h\rangle=\langle Z\otimes Z',\Delta(h)\rangle$.
\end{defn}

Let
$\mathcal H=\oplus_{n\in\mathbb N} \mathcal H_n$ be a graded  connected commutative Hopf algebra of finite type (i.e. each
homogeneous component $\mathcal H_n$ is a finite dimensional vector space).
Let $\mathcal B=\{T_i\}_{i\in\mathbb N}$ be a minimal set of
homogeneous generators of the Hopf algebra $H$ such that
$\deg(T_i)\leq\deg(T_j)$ if $i<j$ and such that $T_0=1$.
For $i>0$, we define the
$\mathbb{C}$-valued infinitesimal character $Z_i$ on generators by
$Z_i(T_j)=\delta_{ij}.$
Let $\mathfrak{g}^{(k)}$
 be the vector space generated by $\{Z_i\ | \ \deg(T_i)\leq k\}$. We
define $\deg(Z_i)=\deg(T_i)$ and set
$$[Z_i,Z_j]_{\mathfrak{g}^{(k)}}=\left\{
\begin{array}{cc}
 [Z_i,Z_j] & \text{if }\deg(Z_i)+\deg(Z_j)\leq k \\
 0 & \text{if }\deg(Z_i)+\deg(Z_j)>k
\end{array}\right.$$
We identify $\varphi\in G_{\mathbb C}$ with
$\{\varphi(T_i)\}\in\mathbb C^{\mathbb N}$ and on $\mathbb
C^{\mathbb N}$ we set a group law given by
$\{\varphi_1(T_i)\}\oplus\{\varphi_2(T_i)\}=\{(\varphi_1\star\varphi_2)(T_i)\}$.
 $G^{(k)}=\{ \{\varphi(T_i)\}_{\{i\, |\, \deg(T_i)\leq
k\}}\ | \ \varphi\in G_{\mathbb C}\}$ is a finite dimensional Lie
subgroup of $G_{\mathbb C}=(\mathbb C^{\mathbb N},\oplus)$ and the Lie algebra
of $G^{(k)}$ is $\mathfrak{g}^{(k)}$.

\begin{defn}[\cite{baditoiu}]
 The double Lie algebra $\delta^{(k)}$ of $\mathfrak{g}^{(k)}$ is the Lie algebra on
 $\mathfrak{g}^{(k)}\oplus\mathfrak{g}^{(k)*}$ with the Lie bracket given by
 $$[X,Y]_{\mathfrak{g}^{(k)}\oplus\mathfrak{g}^{(k)*}}=[X,Y],\ \ \ [X^*,Y^*]_{\mathfrak{g}^{(k)}\oplus\mathfrak{g}^{(k)*}}= 0,\ \ \
 [X,Y^*]=\mathrm{ad}^*_X(Y^*),$$
 for any $X$, $Y\in\mathfrak{g}^{(k)}$, $X^*$, $Y^*\in\mathfrak{g}^{(k)*}$.
\end{defn}
\begin{defn}[\cite{suris}]
Let $G$ be a Lie group with Lie algebra $\mathfrak{g}$.
 A map $f:\mathfrak{g}\to \mathfrak{g}$ is  $\mathrm{Ad}$-{\it covariant} if
$\mathrm{Ad}(g)(f(L))=f(\mathrm{Ad}(g)(L))$
for all $g\in G$, $L\in\mathfrak{g}$.
\end{defn}

Let $\pi:\mathcal A\to \mathcal A$ be a Rota-Baxter projection, which by definition is a linear map with $\pi\circ\pi=\pi$ and satisfying
 the Rota-Baxter equation :
 $$
 \pi(ab)+\pi(a)\pi(b)=\pi(a\pi(b))+\pi(\pi(a)b)
 $$
 for
  $a, b\in\mathcal A.$ For any Rota-Baxter projection $\pi$ on $\mathcal A$, $\mathcal A$ splits into a direct sum of two
subalgebras $\mathcal A=\mathcal A_-\oplus\mathcal A_+$ with $\mathcal A_- = {\rm Im}(\pi)$. In the next theorem we recall the Lax pair equation corresponding to the Connes-Kreimer-Birkhoff factorization of character group of $\mathcal H$ and of a Rota-Baxter projection $\pi$.

\begin{theorem}\label{t:8.2}(\cite[Theorem 5.9+\S 5.4]{baditoiu})
 Let $f:\mathfrak{g}_{\mathcal A}\to\mathfrak{g}_{\mathcal A}$
 be an $\mathrm{Ad}$-covariant
 map.
 Let $L_0\in\mathfrak{g}_{\mathcal A}$ with
 $[f(L_0),L_0]=0$.
 Set $X=f(L_0)$.
 Then the solution of
 \begin{equation}\label{8:1}
        \frac{dL}{dt}=[L,M], \ \ \ M=\frac{1}{2}
        R(f(L))
   \end{equation}
   with initial condition $L(0)=L_0$ is given by
   \begin{equation}\label{e:8.3}
        L(t)=\mathrm{Ad}_{G}g_{\pm}(t)\cdot L_0,
   \end{equation}
   where $\exp(-tX)$ has the Connes-Kreimer Birkhoff factorization
        $\exp(-tX)=g_-(t)^{-1}g_+(t)$.
\end{theorem}
Let $\mathcal A=\mathbb C[\lambda^{-1},\lambda]]$, and $\pi$ be the minimal subtraction scheme. Then the map $f:\mathfrak{g}_\mathcal{A}\to\mathfrak{g}_\mathcal{A}$,
given by $f(L)=2\lambda^{-n+2m}L$,
   is $\mathrm{Ad}$-covariant and
$[f(L_0),L_0]=0$.

\begin{corollary}\label{tc:8.2}(\cite[Corollary 5.11]{baditoiu})
Let $L_0\in\mathfrak{g}_{\mathcal A}$ (with $\mathcal A=\mathbb C[\lambda^{-1},\lambda]]$ and $\pi$  the minimal subtraction scheme)
and set
$X=2\lambda^{-n+2m}L_0$. Then the solution of
 \begin{equation}\label{8:1gen}
        \frac{dL}{dt}=[L,M], \ \ \ M=
        R(\lambda^{-n+2m} L)
   \end{equation}
   with initial condition $L(0)=L_0$ is given by
   \begin{equation}\label{e:8.3gen}
        L(t)=\mathrm{Ad}_{ G_\mathcal{A}}g_{\pm}(t)\cdot L_0,
   \end{equation}
   where $\exp(-tX)$ has the Connes-Kreimer Birkhoff factorization
        $\exp(-tX)=g_-(t)^{-1}g_+(t)$.
 \end{corollary}

\begin{defn}
For $s\in\mathbb C$ and $\varphi\in G_{\mathcal A}$, we define $\varphi^s(x)$ for a homogeneous $x\in H$ by
$$\varphi^s(x)(\lambda)=e^{s\lambda|x|}\varphi(x)(\lambda),$$
for $\lambda\in\mathbb C$ where $|x|$ is the degree of $x$. Let
\begin{equation}\label{locality}
G^{\Phi}_{\mathcal A}=\{\varphi\in G_{\mathcal A}\ \big|\ \frac{ \ \ d}{ds}
(\varphi^s)_-=0\},
\end{equation}
be the set of local characters.

\begin{defn}[\cite{man}] \label{def:6.1}
Let $Y$ be the  biderivation on $\mathcal H=\bigoplus\limits_{n}\mathcal H_n$ given
on homogeneous elements by
$$Y:\mathcal H_n\to \mathcal H_n,\ \ \ \ \ Y(x)=nx\ \ \ \text{for }
x \in\mathcal H_n.$$

 We define the bijection $\tilde R:G_{\mathcal A}\to\mathfrak g_{\mathcal A}$ by
 $$\tilde R(\varphi)=\varphi^{-1}\star(\varphi\circ Y).$$
\end{defn}
 From \cite{baditoiu}, we recall that the locality of the Lax pair flow is preserved.
 \begin{theorem}[\cite{baditoiu}]\label{t:7.9}
Let $\varphi\in G^\Phi_\mathcal{A}$ and
  let $L(t)$ be the solution of the Lax pair equation \eqref{8:1}
  for any $\mathrm{Ad}$-covariant function $f$,
  with the initial condition $L_0=\tilde R(\varphi)$. Let
  $\varphi_t$ be the flow given by
\begin{equation}\label{e:varphit}
\varphi_t=\tilde R^{-1}(L(t)).
\end{equation}
  Then $\varphi_t$ is a local character for all $t$.
\end{theorem}

\end{defn}

\begin{defn}[\cite{baditoiu}]
  For $\varphi\in G_{\mathcal A}^{\Phi}$, $x\in H$, set
  $$\tilde{\beta}_\varphi(x)(\lambda) ={ d\over ds}\Big|_{s=0}
  (\varphi^{-1}\star\varphi^s)(x)(\lambda).$$
\end{defn}
\begin{remark}
For any $\varphi\in G^\Phi_{\mathcal A}$, by  \cite[Lemma 6.6]{baditoiu},  $\tilde\beta_\varphi$ is a holomorphic infinitesimal character  (i.e. $\tilde\beta_\varphi(x)\in\mathcal A_+$ for any $x$) and the relation to the celebrated beta-function of the quantum field theory is given by
\begin{equation}
\beta_\varphi=\mathrm{Ad}(\varphi_+(0))(\tilde\beta_\varphi\big|_{\lambda=0}),
\end{equation}
where $\varphi=\varphi^{-1}_-\star\varphi_+$ is the Birkhoff decomposition of $\varphi$.
\end{remark}

Assuming that $\mathcal A=\mathbb C[\lambda^{-1},\lambda]]$, for $\varphi_t$ given by \eqref{e:varphit}, let the Taylor expansion of
$\tilde\beta_{\varphi_t}$  be
 $$\tilde\beta_{\varphi_t}=\sum_{k=0}^{\infty} \tilde\beta_k(t)\lambda^k.$$

\begin{theorem}[\cite{baditoiu}]\label{t:corr}  For a local character $\varphi\in G^\Phi_\mathcal{A}$,
let $L(t)$ be the Lax pair flow of Corollary  \ref{tc:8.2} with $\mathcal A=\mathbb C[\lambda^{-1},\lambda]]$ and with  initial
condition $L_0=\tilde R(\varphi)$. Let $\varphi_t=\tilde
R^{-1}(L(t))$. Then
\begin{itemize}
\item[(i)] for $-n+2m\geq 1$,  $\varphi_t = \varphi$ and hence
  $\beta_{\varphi_t}=\beta_{\varphi}$ and $\tilde\beta_0(t)=\tilde\beta_0(0)$ for all $t$.
\item[(ii)] for $-n+2m\leq 0$,  $\beta_{\varphi_t}\in
\mathfrak{g}_{\mathbb C}$ satisfies
\begin{eqnarray}\label{e:t:corr:betazero}
\frac{d\tilde\beta_0(t)}{dt}
=2[\tilde\beta_0(t),\tilde\beta_{n-2m+1}(t)].
\end{eqnarray}
\end{itemize}
\end{theorem}

\section{Completely integrable Lax pair examples for $\tilde\beta_0$}\label{main}
Under set up of Theorem \ref{t:corr}
and we give explicitly the integrals of motion the equation \eqref{e:t:corr:betazero}
 on the double Lie algebra $\delta$ of certain
particular truncated Lie algebras.
We always assume that we are in the nontrivial case of Theorem \ref{t:corr},
namely that $-n+2m\leq 0$. The gauge transformation
$\tilde\beta_0(t)\to (\varphi_t)_+(0)\star
\tilde\beta_0(t)\star\left((\varphi_t)_+(0)\right)^{-1}$ changes
equation \eqref{e:t:corr:betazero} into the equation of beta functions.

Let $\mathcal H_1$ be the Hopf subalgebra of rooted trees generated
by
 $$t_0=1_\mathcal{T}, \;\;\;\;
 t_1= \ta1,           \;\;\;\;
 t_2= \tb2,           \;\;\;\;
 t_4=\td31,
 $$
 and let $G_1$ and $\mathfrak g_1$ be its group of character with values in $\mathbb C$ and
 respectively
 its Lie algebra of infinitesimal characters.
Identifying $G_1\equiv\mathbb C^3$ via
$\varphi\to{\varphi(f_i)}_{i=1,2,4}$, with $\{f_i\}_{i=1,2,4}$ the
normal coordinates  (\cite{chr}) associated to $t_1,t_2,t_4$:
$$
   f_1=\ta1\; , \;\;\;\; f_2=\tb2-\frac{1}{2}\ta1\ta1 \; ,\;\;\;\;
   f_4=\td31-\ta1\tb2+\frac{1}{6} \ta1^3\;,
$$
we get $G_1$ is exactly the 3-dimensional Heisenberg group
 $(\mathbb C^3,\oplus)$, where
$$(x_1,x_2,x_3)\oplus(y_1,y_2,y_3)=(x_1+y_1,x_2+y_2,x_3+y_3+x_1y_2-x_2y_1),$$
and its Lie algebra $\mathfrak g_1=\mathrm{Span}(X_1,X_2,X_3)$ is
the 2-step nilpotent given  by $[X_1,X_2]=2X_3$ and $[X_i,X_j]=0$
for any $(i,j)\not=(1,2)$ and $(i,j)\not=(2,1)$.

Since there is no Ad-invariant, non-degenerate, symmetric bilinear
form on $\mathfrak g_1$, we pass to its double Lie algebra
$\delta_1=\mathfrak g_1\oplus\mathfrak g_1^*$ which has a natural
pairing $\langle\cdot,\cdot\rangle$. The natural Lie-Poisson bracket
of $\delta_1^*$ rises to a Poisson structure on $\delta_1$ via
$\langle\cdot,\cdot\rangle$. The nontrivial Lie brackets of
$\delta_1=\mathrm{Span}(X_1,X_2,X_3,X_1^*,X_2^*,X_3^*)$ are
$$[X_1,X_2]=2X_3,\ \ [X_1,X_3^*]=-2X_2^*,\ \  [X_2,X_3^*]=2X_1^*,$$
while its associated Lie-Poisson bracket on $\delta_1$ is determined
by
$$\{x_1,x_2\}(x)=2x_3^*, \ \ \{x_1,x_3^*\}(x)=-2x_2, \ \
\{x_2,x_3^*\}(x)=2x_1,$$ with
 $x=\sum_{i=1}^{i=3} x_iX_i+\sum_{i=1}^{i=3}x_i^*X_i^*. $
Straightforward computations give the following lemma.
\begin{lemma}\label{l:h1h2h3h4h5}
(i) Both $\mathfrak{g}_1$ and $\delta_1$  are 2-step nilpotent Lie
algebras, i.e. $[\mathfrak{g}_1,[\mathfrak{g}_1,\mathfrak{g}_1]]=0,$
$[\delta_1,[\delta_1,\delta_1]]=0.$

(ii) The functions
$$
 H_1(x)=\frac{x_1^2}{2}+\frac{x_2^2}{2},\ \
 H_2(x)=\frac{x_3^2}{2},\ \
 H_3(x)=\frac{x_1^{*2}}{2},\ \
 H_4(x)=\frac{x_2^{*2}}{2},\ \
 H_5(x)=\frac{x_3^{*2}}{2}+\frac{x_1^2}{2}+\frac{x_2^2}{2},
$$
are in involution with respect to the Lie-Poisson bracket on
$\delta_1$ and $dH_1$, $dH_2$, $dH_3$, $dH_4$, $dH_5$ are linearly
independent on an open dense set in $\delta_1$.

(iii) $\delta_1$ is a Poisson manifold of rank $2r=2$.
\end{lemma}

\begin{proof} (iii)
The adjoint representation
$\mathrm{ad}:\delta_1\to\mathfrak{gl}(\delta_1)$,
$$ \mathrm{ad}_{\delta_1}(x)=
{\small \left(
\begin{array}{cccccc}
0&      0&      0&  0& 0& 0\\
0&      0&       0&  0& 0& 0\\
-2x_2 & 2 x_1  & 0&  0& 0& 0\\
0     &-2 x_3^*& 0& 0& 0& 2 x_2\\
2x_3^*& 0& 0& 0&  0 &-2 x_1\\
0& 0& 0& 0 & 0& 0
\end{array}
\right)}
$$
has the maximal rank 2, thus the rank $2r$ of the Poisson structure is
also $2$.
\end{proof}
In order to show that equation \eqref{e:t:corr:betazero} is integrable on $\delta_1$ it is sufficient to find a function
$H:\delta_1\to\mathbb R$ such $\tilde\beta_{n-2m+1}(t)=\nabla
H(\tilde\beta_0(t))$ (i.e $H$ is a Hamiltonian for \eqref{e:t:corr:betazero})
and to show that the functions $H$, $H_2$, $H_3$, $H_4$, $H_5$ are
in involution and linearly independent (in the sense that their
differentials are linearly independent on an open dense set). Here
$\nabla H(x)$ denotes the gradient of $H(x)$. The idea is to show
that the $2$-step nilpotency of $\mathfrak g_1$ implies that both
$\tilde\beta_{n-2m+1}(t)$ and $\tilde\beta_0(t)$ are linear in the
variable $t$.
\begin{lemma}\label{l:thehamiltonian}
There exists a Hamiltonian function of the form
$$H(x)=k_1x_1+k_2x_2+k_3x_3+\frac{l_1x^2_1}{2}+\frac{l_2x^2_2}{2}+\frac{l_3x^2_3}{2}$$
such that $\tilde\beta_{n-2m+1}(t)=\nabla H(\tilde\beta_0(t))$.
\end{lemma}
\begin{proof}
Differentiating \eqref{e:t:corr:betazero} and then using the 2-step
nilpotency of $\delta_1$, we get that
\begin{eqnarray}\label{betazerodiff}
\frac{d^2\tilde\beta_0(t)}{dt^2}&=&[[\tilde\beta_0(t),\tilde\beta_{n-2m+1}(t)],\tilde\beta_{n-2m+1}(t)]
 +[\tilde\beta_0(t),\frac{d\tilde\beta_{n-2m+1}(t)}{dt}]\nonumber\\
 &=&[\tilde\beta_0(t),\frac{d\tilde\beta_{n-2m+1}(t)}{dt}].
\end{eqnarray}
By the proof of Theorem 8.7 in \cite{baditoiu}, $\tilde\beta_{\varphi_t}$ satisfies the
equation
$$\frac{d\tilde\beta_{\varphi_t}}{dt}=-2\left[\sum_{k=n-2m+1}^\infty
\tilde\beta_{k}(t)\lambda^{k-n+2m-1},\sum_{j=0}^{n-2m}\tilde\beta_j(t)\lambda^j\right],
$$
thus the coefficient $\tilde\beta_{n-2m+1}(t)$ corresponding to the
power $\lambda^{n-2m+1}$ in the Taylor expansion of
$\tilde\beta_{\varphi_t}$ satisfies the relation:
\begin{eqnarray}\label{betan2m1diff}
\frac{d\tilde\beta_{n-2m+1}(t)}{dt}&=&-2\sum_{k=0}^{n-2m}
 [\tilde\beta_{n-2m+2+k}(t),\tilde\beta_{n-2m-k}(t)] ,
\end{eqnarray}
which implies
\begin{eqnarray}\label{betan2m1diff}
\frac{d^2\tilde\beta_{n-2m+1}(t)}{dt^2}&=&-2\sum_{k=0}^{n-2m}
\left(\frac{d}{dt}[\tilde\beta_{n-2m+2+k}(t),\tilde\beta_{n-2m-k}(t)]
\right) \\
&=&-2\sum_{k=0}^{n-2m}\left(-2\sum_{j=0}^{n-2m}[[\tilde\beta_{n-2m+3+k+j}(t),\tilde\beta_{n-2m-j}(t)]
 ,\tilde\beta_{n-2m-k}(t)]\right.\nonumber\\
 &&\left.-2\sum_{j=0}^{n-2m}[\tilde\beta_{n-2m+2+k}(t),[\tilde\beta_{n-2m-k+1+j}(t),\tilde\beta_{n-2m-k-j}(t)]]
\right)\nonumber
\end{eqnarray}
By equations \eqref{e:t:corr:betazero}, \eqref{betazerodiff},
\eqref{betan2m1diff}, the 2-step nilpotency of $\mathfrak g_1$
implies
$$\frac{d^2\tilde\beta_0(t)}{dt^2}=0\ \text{and}\
\frac{d^2\tilde\beta_{n-2m+1}(t)}{dt^2}=0,$$ thus
$\tilde\beta_0(t)=at+b$ and $\tilde\beta_{n-2m+1}(t)=ct+d$ for some
$a,b,c,d\in\mathfrak g_1$. If $
H(x)=k_1x_1+k_2x_2+k_3x_3+\frac{l_1x^2_1}{2}+\frac{l_2x^2_2}{2}+\frac{l_3x^2_3}{2}
$ then its gradient is $\nabla H(x)=\sum_{k=1}^3(k_i+l_ix_k)X_i$.
Obviously, there exists a solution $\{k_1,k_2,k_3,l_1,l_2,l_3\}$ of
system of equations $c_i+td_i=k_i+l_i(a_i+tb_i),\ i\in\{1,2,3\}$ for
any $t$.
\end{proof}

\begin{lemma}\label{l:hh2h3h4h5}
The functions $H$, $H_2$, $H_3$, $H_4$, $H_5$ are in involution and
linearly independent.
\end{lemma}
\begin{proof}
Since the Poisson bracket between any of the coordinates $x_1^{*},
x_2^*, x_3$ and any coordinate $x_i$ or $x_i^*$, $i\in\{1,2,3\}$
vanishes, it follows that any of $H_2$, $H_3$, $H_4$ is in
involution with any of $H$, $H_2$, $H_3$, $H_4$, $H_5$. We
additionally have
\begin{eqnarray*}
\{H,H_5\}&=&
 \left(
 \frac{\partial H}{\partial x_1}
 \frac{\partial H_5}{\partial x_2}
 -\frac{\partial H}{\partial x_1}
 \frac{\partial H_5}{\partial x_2}\right)
 \{x_1,x_2\}
 + \frac{\partial H}{\partial x_1}
 \frac{\partial H_5}{\partial x_3^*}
 \{x_1,x_3^*\}
 +\frac{\partial H}{\partial x_2}
 \frac{\partial H_5}{\partial x_3^*}
 \{x_2,x_3^*\}\\
 &=&\big((k_1+l_1x_1)x_2-(k_2+l_2x_2)x_1\big)(2x_3^*)+(k_1+l_1x_1)x_3^*(-2x_2)+
(k_2+l_2x_2)x_3^*(2x_1)\\
 &=&0.
\end{eqnarray*}
Since the Jacobian matrix
$$ \frac{\partial(H,H_2,H_3,H_4,H_5)}{\partial(x_1,x_2,x_3,x_1^*,x_2^*,x_3^*)}=
{\small \left(
\begin{array}{cccccc}
k_1+l_1x_1&k_2+l_2x_2&k_3+l_3x_3&  0& 0& 0\\
0& 0 &x_3&  0& 0& 0\\
0& 0 & 0 &x_1^*&  0 & 0\\
0& 0 & 0 & 0 &x_2^*& 0\\
x_1&x_2&0& 0 & 0 &x_3^*
\end{array}
\right)}
$$
has the rank 5 on an open dense set, we get that
 $dH$, $dH_2$, $dH_3$, $dH_4$, $dH_5$
are linearly independent on that set.
\end{proof}
By Lemmas \ref{l:h1h2h3h4h5}(iii), \ref{l:thehamiltonian},
\ref{l:hh2h3h4h5}, we conclude the following theorem.

\begin{theorem}
The equation \eqref{e:t:corr:betazero} is integrable on $\delta_1$.
\end{theorem}

Similarly to the $\delta_1$ case, we can prove that the equation
\eqref{e:t:corr:betazero} is integrable for a 3-step nilpotent Lie
algebra and Hamiltonian for a 4-step nilpotent Lie algebra. We
consider to the Hopf subalgebra of rooted trees $\mathcal H_2$
generated by
$$ 1_\mathcal{T}, \;\;\;\;
  \ta1,           \;\;\;\;
   \tb2,           \;\;\;\;
  \td31,   \;\;\;\;
\th43
 $$
\def\tw5{{\scalebox{0.25}{
\begin{picture}(42,42)(8,-8)
 \SetWidth{0.5} \SetColor{Black}
 \Vertex(14,-3){5.66}
 \Vertex(37,27){5.66} 
 \Vertex(60,-3){5.66}
 \SetWidth{1.0}
 \Line(14,-3)(37,27)
 \Line(44,-3)(37,27)
 \Line(37,27)(30,-3)
 \Line(37,27)(60,-3)
 \SetWidth{0.5}
 \Vertex(30,-3){5.66}
 \Vertex(44,-3){5.66}
\end{picture}}}}
 and  the Hopf subalgebra $\mathcal H_3$ generated by
$$ 1_\mathcal{T}, \;\;\;\;
 \ta1,           \;\;\;\;
 \tb2,           \;\;\;\;
 \td31,   \;\;\;\;
 \th43,   \;\;\;\;
 \tw5.
 $$
Let $\mathfrak g_2$ and $\mathfrak g_3$ be the  Lie algebras of
infinitesimal characters with values in $\mathbb C$ of $\mathcal
H_2$ and  respectively $\mathcal H_3$. Let $\delta_2$ and $\delta_3$
the double Lie algebras of  $\mathfrak g_2$ and $\mathfrak g_3$. For
a nonempty tree $T_1$, let $Z_{T_1}$ be the infinitesimal character
determined on trees by $Z_{T_1}(T_2)=1$ if $T_1=T_2$ and
$Z_{T_1}(T_2)=0$ if $T_1\not=T_2$. Let
 $$X_1=Z_\ta1,  \;\;\;\; X_2=Z_\tb2,  \;\;\;\; X_3=Z_\td31,   \;\;\;\;
 X_4=Z_\th43,   \;\;\;\;
 X_5=Z_\tw5,
 $$
and $X_1^*, X_2^*,X_3^*, X_4^*, X_5^*$ the dual base of $X_1 , X_2
,X_3 , X_4 , X_5$.
 The Lie algebras of $\delta_2$ and  $\delta_3$ have the following nonzero Lie
 brackets:
$$[X_1,X_2]_{\delta_2}=2X_3,\ \ [X_1,X_3^*]_{\delta_2}=-2X_2^*,\ \  [X_2,X_3^*]_{\delta_2}=2X_1^*,$$
$$[X_1,X_3]_{\delta_2}=3X_4,\ \ [X_1,X_4^*]_{\delta_2}=-3X_3^*,\ \  [X_3,X_4^*]_{\delta_2}=3X_1^*,$$
$$[X_1,X_2]_{\delta_3}=2X_3,\ \ [X_1,X_3^*]_{\delta_3}=-2X_2^*,\ \  [X_2,X_3^*]_{\delta_3}=2X_1^*,$$
$$[X_1,X_3]_{\delta_3}=3X_4,\ \ [X_1,X_4^*]_{\delta_3}=-3X_3^*,\ \  [X_3,X_4^*]_{\delta_3}=3X_1^*,$$
$$[X_1,X_4]_{\delta_3}=4X_5,\ \ [X_1,X_5^*]_{\delta_3}=-4X_4^*,\ \  [X_4,X_5^*]_{\delta_3}=4X_1^*,$$
 Notice that
$\delta_2$  is not a Lie subalgebra of $\delta_3$, just because of
$[X_1,X_4]_{\delta_2}=0$ and $[X_1,X_4]_{\delta_3}=4X_5$, but
$\delta_2$ is a truncation of $\delta_3$. Notice also that $\mathfrak g_2$ and $\delta_2$ are step
3-nilpotent Lie algebras, while $\mathfrak g_3$ and $\delta_3$ are
step 4-nilpotent Lie algebras. By an argument  similar to the one
above, the 3-step nilpotency of $\mathfrak g_2$ implies that
$\tilde\beta_0(t)$ and $\tilde\beta_{n-2m+1}(t)$ are quadratic in
$t$ when we consider equation \eqref{e:t:corr:betazero} on $\mathfrak g_2$,
while the 4-step nilpotency of $\mathfrak g_3$ implies that
$\tilde\beta_0(t)$ and $\tilde\beta_{n-2m+1}(t)$ are cubic in $t$
for the Lie algebra $\mathfrak g_3$.

The Lie-Poisson brackets are nonzero only for the following pairs:
$$\{x_1,x_2\}_{\delta_2}(x)=2x_3^*, \ \ \{x_1,x_3^*\}_{\delta_2}(x)=-2x_2, \ \
\{x_2,x_3^*\}_{\delta_2}(x)=2x_1,$$
$$\{x_1,x_3\}_{\delta_2}(x)=3x_4^*, \ \ \{x_1,x_4^*\}_{\delta_2}(x)=-3x_3, \ \
\{x_3,x_4^*\}_{\delta_2}(x)=3x_1,$$
$$\{x_1,x_2\}_{\delta_3}(x)=2x_3^*, \ \ \{x_1,x_3^*\}_{\delta_3}(x)=-2x_2, \ \
\{x_2,x_3^*\}_{\delta_3}(x)=2x_1,$$
$$\{x_1,x_3\}_{\delta_3}(x)=3x_4^*, \ \ \{x_1,x_4^*\}_{\delta_3}(x)=-3x_3, \ \
\{x_3,x_4^*\}_{\delta_3}(x)=3x_1,$$
 $$\{x_1,x_4\}_{\delta_3}(x)=4x_5^*,
\ \ \{x_1,x_5^*\}_{\delta_3}(x)=-4x_4, \ \
\{x_4,x_5^*\}_{\delta_3}(x)=4x_1.$$
 Straightforward computations shows that $(\delta_2,
 \{\cdot,\cdot\}_{\delta_2})$ and $(\delta_3,
 \{\cdot,\cdot\}_{\delta_3})$ are Poisson manifolds of ranks $2r_2=4$
 and respectively $2r_2=6$. The following functions defined on
 $\delta_2$ are in involution with respect to $\{\cdot,\cdot\}_{\delta_2}$
  and linearly independent:
$$
 H^{\delta_2}_1(x)=\frac{x_1^2}{2}+\frac{x_3^2}{2},\ \ \
 H^{\delta_2}_2(x)=\frac{x_4^2}{2},\ \ \
 H^{\delta_2}_3(x)=\frac{x_1^{*2}}{2},\ \ \
 H^{\delta_2}_4(x)=\frac{x_2^{*2}}{2},
$$ $$
H^{\delta_2}_5(x)=\frac{x_4^{*2}}{2}+H^{\delta_2}_1(x) ,\ \ \
 H^{\delta_2}_6(x)=\frac{x_2^2}{2}+\frac{x_3^{*2}}{2}+H^{\delta_2}_5(x).
$$
On $\delta_3$, the following functions  are in involution with
respect to $\{\cdot,\cdot\}_{\delta_3}$
  and linearly independent:
$$
 H^{\delta_3}_1(x)=\frac{x_1^2}{2}+\frac{x_4^2}{2},\ \ \
 H^{\delta_3}_2(x)=\frac{x_5^2}{2},\ \ \
 H^{\delta_3}_3(x)=\frac{x_1^{*2}}{2},\ \ \
 H^{\delta_3}_4(x)=\frac{x_2^{*2}}{2},
$$
$$
 H^{\delta_3}_5(x)=\frac{x_2^2}{2}+\frac{x_3^{*2}}{2},\ \ \
 H^{\delta_3}_6(x)=\frac{x_3^2}{2}+\frac{x_4^{*2}}{2},\ \ \
 H^{\delta_3}_7(x)=\frac{x_5^{*2}}{2}.
$$

Now, we show that the equation \eqref{e:t:corr:betazero} is a Hamiltonian
equation both for  $ \mathfrak g_2 $, $ \mathfrak g_3 $. Let
$H^{\mathfrak g_2}$, $H^{\mathfrak g_3}$ be quadratic functions on
$\mathfrak g_2$ and respectively on $\mathfrak g_3$. The function
 $$H^{\mathfrak g_2}(x)=\sum_{i=1}^{4}k_ix_i+\frac{l_ix_i^2}{2}+\sum_{1\leq
 j<p\leq 4}\xi_{j,p}x_jx_p$$
is determined by 14 variables
 $\{k_i,l_i,\xi_{j,p}\}_{1\leq i\leq 4, \ 1\leq  j<p\leq 4}
$, while a quadratic function on $ \mathfrak g_3$ is determined by
20 variables. In local coordinates, after identifying the
coefficients in t and taking into the account that
$\tilde\beta_0(t)$, $\tilde\beta_{n-2m+1}(t)$ are quadratic in t
(for the $\mathfrak g_2$ case) or cubic (for the $\mathfrak g_3$
case), the equation $\nabla H^{\mathfrak g_2}
(\tilde\beta_0(t))=\tilde\beta_{n-2m+1}(t)$ reduces to  a systems of
12 linear equations, and the equation $\nabla H^{\mathfrak g_3}
(\tilde\beta_0(t))=\tilde\beta_{n-2m+1}(t)$ reduces to  a systems of
20 linear equations, thus in both cases, there exists a Hamiltonian
function for \eqref{e:t:corr:betazero}.

For $i\in\{2,3\}$, let
 $\pi_i:\delta_i=\mathfrak g_i\oplus \mathfrak g_i^*\to \mathfrak g_i$
 be the projections onto $\mathfrak g_i$ and let $H^{\delta_i}= H^{\mathfrak g_i}\circ
 \pi_i$.
 The functions $H^{\delta_2}$, $H^{\delta_2}_2$,
$H^{\delta_2}_3$, $H^{\delta_2}_4$, $H^{\delta_2}_5$,
$H^{\delta_2}_6$ are in involution with respect to
$\{\cdot,\cdot\}_{\delta_2}$ and linearly independent. This
concludes that equation \eqref{e:t:corr:betazero} is integrable for
${\delta_2}$.

The polynomial functions $H^{\delta_3}$, $H^{\delta_3}_2$, $H^{\delta_3}_3$,
$H^{\delta_3}_4$,
$H^{\delta_3}_1+H^{\delta_3}_5+H^{\delta_3}_6+H^{\delta_3}_7$ are in
involution with respect to $\{\cdot,\cdot\}_{\delta_3}$ and linearly
independent. In order to show that equation \eqref{e:t:corr:betazero} is
integrable with the Hamiltonian $H^{\delta_3}$, it would be
sufficient to find another two functions $f_1, f_2$ on $\delta_3$
such that
 $H^{\delta_3}$, $H^{\delta_3}_2$,
$H^{\delta_3}_3$, $H^{\delta_3}_4$,
$H^{\delta_3}_1+H^{\delta_3}_5+H^{\delta_3}_6+H^{\delta_3}_7$,
$f_1$, $f_2$ are in involution with respect to
$\{\cdot,\cdot\}_{\delta_3}$ and linearly independent.
The existence of such two functions $f_1$ and $f_2$ is ensured by the proved  the Mishchenko-Fomeko conjecture presented in \cite{bol}.

We summarize the following.
\begin{theorem}
(i) If $G_2$ is the groups of characters of the   Hopf subalgebra of rooted trees $\mathcal H_2$, then the equation \eqref{e:t:corr:betazero} is completely
integrable for $\delta_2$. \\
(ii) If $G_3$ is the groups of characters of the   Hopf subalgebra of rooted trees $\mathcal H_2$, then the equation \eqref{e:t:corr:betazero} is completely
integrable for $\delta_3$.
\end{theorem}

\bibliographystyle{amsplain}
\bibliography{int-sys-lmp}

\providecommand{\bysame}{\leavevmode\hbox to3em{\hrulefill}\thinspace}
\providecommand{\MR}{\relax\ifhmode\unskip\space\fi MR }
\providecommand{\MRhref}[2]{%
  \href{http://www.ams.org/mathscinet-getitem?mr=#1}{#2}
}
\providecommand{\href}[2]{#2}
\begin{thebibliography}{1}

\bibitem{baditoiu}
Gabriel B{\u{a}}di{\c{t}}oiu and Steven Rosenberg, \emph{Lax pair equations and
  {C}onnes-{K}reimer renormalization}, Comm. Math. Phys. \textbf{296} (2010),
  no.~3, 655--680. \MR{2628819}

\bibitem{bol}
Alexey Bolsinov, \emph{Complete commutative subalgebras in polynomial {P}oisson
  algebras: a proof of the {M}ischenko--{F}omenko conjecture},
  \href{http://arxiv.org/abs/1206.3882/}{arXiv:1206.3882}.

\bibitem{chr}
C.~Chryssomalakos, H.~Quevedo, M.~Rosenbaum, and J.~D. Vergara, \emph{Normal
  {C}oordinates and {P}rimitive {E}lements in the {H}opf {A}lgebra of
  {R}enormalization}, Comm. Math. Phys. \textbf{225} (2002), 465--485,
  hep--th/0105259.

\bibitem{ck1}
A.~Connes and D.~Kreimer, \emph{Renormalization in {Q}uantum {F}ield {T}heory
  and the {R}iemann-{H}ilbert {P}roblem {I}: {T}he {H}opf algebra structure of
  {G}raphs and the {M}ain {T}heorem}, Comm. Math. Phys. \textbf{210} (2000),
  249--273.

\bibitem{man}
D.~Manchon, \emph{{H}opf algebras, from basics to applications to
  renormalization}, {C}omptes-rendus des {R}encontres math\'ematiques de
  {G}lanon 2001, 2003, math.QA/0408405.

\bibitem{sts}
A.G. Reyman and M.A. Semenov-Tian-Shansky, \emph{{I}ntegrable {S}ystems {II}:
  Group-{T}heoretical {M}ethods in the {T}heory of {F}inite-{D}imensional
  {I}ntegrable {S}ystems}, Dynamical systems. {VII}, Encyclopaedia of
  Mathematical Sciences, vol.~16, Springer-{V}erlag, Berlin, 1994.

\bibitem{suris}
Yuri~B. Suris, \emph{The problem of integrable discretization: {H}amiltonian
  approach}, Progress in Mathematics, vol. 219, Birkh\"auser Verlag, Basel,
  2003.

\end{thebibliography}
\end{document}